\documentclass[12pt]{amsart}
\usepackage[norelsize]{algorithm2e}
\usepackage{amsmath,amssymb,txfonts}
\usepackage{amsfonts}
\usepackage[utf8]{inputenc}
\usepackage[T1]{fontenc}
\usepackage{lmodern}
\usepackage{mathrsfs}
\usepackage{amsthm}
\usepackage{verbatim}
\usepackage{tikz}
\usetikzlibrary{matrix}
\usepackage{caption}

\newcommand{\N}{\mathbb{N}}
\newcommand{\T}{\mathbb{T}}
\newcommand{\D}{\mathbb{D}}
\newcommand{\C}{\mathbb{C}}

\newtheorem{bddmap}{Proposition}[section]
\newtheorem{masslemma}[bddmap]{Lemma}
\newtheorem{normlimit}[bddmap]{Proposition}
\newtheorem{localization}[bddmap]{Lemma}
\newtheorem{localization2}[bddmap]{Lemma}
\newtheorem{isomorphism}[bddmap]{Theorem}
\newtheorem{T_gstrictsing}[bddmap]{Theorem}
\newtheorem{aleman_cima}[bddmap]{Theorem}

\newtheorem{lemma}[bddmap]{Lemma}
\date{\today}
\begin{document}

\author{Santeri Miihkinen}
\address{Santeri Miihkinen:\ Department of Mathematics and Statistics, University of Helsinki, Box 68, 00014 Helsinki, Finland}
         \email{santeri.miihkinen@helsinki.fi}
\title[\resizebox{4.5in}{!}{Strict singularity of a Volterra-type integral operator on $H^p$}]{Strict singularity of a Volterra-type integral operator on $H^p$}

\begin{abstract}
We prove that a Volterra-type integral operator $$T_gf(z) = \int_0^z f(\zeta)g'(\zeta)d\zeta, \quad z \in \D,$$ defined on Hardy spaces $H^p, \, 1 \le p < \infty,$ fixes an isomorphic copy of $\ell^p,$ if the operator $T_g$ is not compact. In particular, this shows that the strict singularity of the operator $T_g$ coincides with the compactness of the operator $T_g$ on spaces $H^p.$ As a consequence, we obtain a new proof for the equivalence of the compactness and the weak compactness of the operator $T_g$ on $H^1$.   
\end{abstract}
\maketitle
\section{Introduction}

 Let $g$ be a fixed analytic function in the open unit disc $\D$ of the complex plane $\C$. We consider a linear integral operator $T_g$ defined formally for analytic functions $f$ in $\D$ by $$T_gf(z) = \int_0^z f(\zeta)g'(\zeta)d\zeta, \quad z \in \D.$$ 
Ch.\ Pommerenke was the first author to consider the boundedness of the operator $T_g$ on Hardy space $H^2$ and he characterized it in \cite{Pommerenke} in a connection to exponentials of $BMOA$ functions. 
A systematic study of the operator $T_g$ was initiated by A.\ Aleman and A.\ G.\ Siskakis in \cite{AS1}, where they stated the boundedness and compactness characterization of $T_g$ on Hardy spaces $H^p, \, 1 \le p < \infty.$ Namely, they observed that $T_g$ is bounded (compact) if and only if $g \in BMOA$ ($g \in VMOA$). The same boundedness characterization of the operator $T_g$ on $H^p, \, 0 < p < 1,$ spaces was obtained by Aleman and J.\ Cima in \cite{AC}. 
Many properties of the operator $T_g$ have been studied by several authors later on and they are well known in most spaces of analytic functions, see also surveys \cite{A2} and \cite{Siskakis}. 

However, one operator theoretically interesting property, the strict singularity, has not been considered in the case of $T_g.$ 
A bounded operator $S\colon X \to Y$ between Banach spaces is strictly singular if its restriction to any infinite-dimensional closed subspace is not an isomorphism onto its image. This notion generalizes the concept of compact operators and it was introduced by T.\ Kato in \cite{Kato}. Canonical examples of strictly singular non-compact operators are inclusion mappings $i_{p,q} \colon \ell^p \hookrightarrow \ell^q,$ where $1 \le p < q < \infty.$ There also exist strictly singular non-compact operators on $H^p$ spaces for $1 \le p < \infty, \, p \ne 2$. 

The aim of this note is to show that a non-compact operator $T_g$ defined on Hardy spaces $H^p, \, 1 \le p < \infty,$ fixes an isomorphic copy of $\ell^p.$ In particular, this implies that the operator $T_g$ is strictly singular on $H^p$ if and only if it is compact. 
Moreover, this gives a new proof for the equivalence of compactness and weak compactness of $T_g$ on Hardy space $H^1,$ see \cite{LMN}.     

Our main result is the following theorem.
\begin{T_gstrictsing}
\label{T_gstrictsing}
Let $g \in BMOA \setminus VMOA$ and $1 \le p < \infty.$ Then the operator $$T_g \colon H^p \to H^p$$ fixes an isomorphic copy of $\ell^p$ inside $H^p.$ In particular, the operator $T_g$ is not strictly singular, i.e.\ the class of strictly singular operators $T_g$ coincides with the class of compact operators $T_g$.
\end{T_gstrictsing}

We should point out that there is an interesting extrapolation result by Hern{\'a}ndez,  Semenov, and Tradacete\ in \cite[Theorem 3.3]{Hernandez}. It states that if an operator $S$ is bounded on $L^p$ and $L^q$ for some $1 < p < q < \infty$ and strictly singular on $L^r$ for some $p < r < q,$ then it is compact on $L^s$ for all $p < s < q.$ If the corresponding statement for $L^p$ spaces of complex-valued functions is true, then the equivalence of strict singularity and compactness of $T_g$ on $H^p$ for $1 < p < \infty$ follows immediately by using the Riesz projection: 
Recall that strictly singular operators form a two-sided (closed) ideal in the space $\mathcal{L}(L^p)$ of bounded operators on $L^p=L^p(\T),$ where $\T = \partial\D$. Therefore the strict singularity of $T_g\colon H^p \to H^p$ implies that $T_gR\colon L^p \to L^p$ is strictly singular, where $R\colon L^p \to H^p$ is the Riesz projection and we have identified $T_g\colon H^p \to H^p$ with $T_g\colon H^p \to L^p$. Since the condition $g \in BMOA$ characterizes the boundedness of $T_g$ on every $H^q, \, 0 < q < \infty,$ space and the Riesz projection is bounded on the scale $1 < q < \infty$, we get that $T_gR$ is bounded on every $L^q, \, 1 < q < \infty,$ space. Now assuming that the complex version of the interpolation result is valid, it follows that $T_gR$ is compact on $L^p$ and consequently the restriction $T_gR|_{H^p} = T_g$ is compact on $H^p.$

However, Theorem \ref{T_gstrictsing} states more: a non-compact operator $T_g$ on $H^p$ fixes an isomorphic copy of $\ell^p$ and this is also true in the case $p = 1.$

Theorem \ref{T_gstrictsing} also gives a new proof for the equivalence of the compactness and the weak compactness of the operator $T_g$ on $H^1:$ If $g \in BMOA \setminus \nolinebreak VMOA,$ i.e.\ the operator $T_g$  is not compact, then by Theorem  \ref{T_gstrictsing} the operator $T_g$ fixes an infinite-dimensional subspace $M$, an isomorphic copy of $\ell^1.$ The class of compact operators on $\ell^1$ coincides with the class of weakly compact operators on $\ell^1$. As an isomorphism, the restriction $T_g|_M$ is not compact and hence it is not weakly compact. Therefore the operator $T_g$ is not weakly compact.

\section{Preliminaries}

In this section, we briefly remind a reader some common spaces of analytic functions that appear later and state a theorem of Aleman and Cima which we need in the proof our main result  Theorem \ref{T_gstrictsing}. 
 
Let $H(\D)$ be the algebra of analytic functions in $\D$. 
We define Hardy spaces 
$$H^p = \left\{f \in H(\D): \|f\|_p = \left( \sup_{0 \le r < 1}\frac{1}{2\pi}\int_0^{2\pi} |f(re^{it})|^p dt\right)^{1/p} < \infty \right\}.$$ Space $BMOA$ consists of functions $f \in H(\D)$ with
$$\|f\|_* = \sup_{a \in \D}\|f \circ \sigma_a - f(a)\|_2 < \infty,$$ where $\sigma_a(z) = (a-z)/(1-\bar{a}z)$ is the Möbius automorphism of $\D$ that interchanges the origin and the point $a \in \D$. Its closed subspace $VMOA$ consists of those $f \in H(\D)$ with $$\limsup_{|a| \to 1}\|f \circ \sigma_a - f(a)\|_2 = 0.$$ See e.g.\ \cite{Girela} for more information on spaces $BMOA$ and $VMOA$. The Bloch space $\mathcal{B}$ is the Banach space of functions $f \in H(\D)$ s.t.\ $$\sup_{z \in \D}(1-|z|^2)|f'(z)| < \infty.$$

We use notation $A \lesssim B$ to indicate that $A \le cB$ for some positive constant $c$ whose value may change from one occurence into another and which may depend on $p$. If $A \lesssim B$ and $B \lesssim A,$ we say that the quantities $A$ and $B$ are equivalent and write $A \simeq B.$ 

Every $BMOA$ function $f$ satisfies a reverse ``Hölder's inequality'', which implies that for each $0 < p < \infty$ it holds that $$\|f\|_* \simeq \sup_{a \in \D}\|f \circ \sigma_a - f(a)\|_p < \infty,$$ where the proportionality constants depend on $p.$ Similarly, a function $f$ is in $VMOA$ if and only if $$\limsup_{|a| \to 1}\|f \circ \sigma_a - f(a)\|_p = 0.$$ 

The proof of Theorem \ref{T_gstrictsing} utilizes a result of Aleman and Cima \cite[Theorem 3]{Aleman_Cima}. We state it here for convenience.
\begin{aleman_cima}
\label{aleman_cima}
Let $p > 0$ and $g \in H^p.$ For $a \in \D,$ let $\sigma_a(z) = (a - z)/(1 - \bar{a}z)$ and  
$f_a(z) = (1 - |a|^2)^{1/p}/(1 - \bar{a}z)^{2/p}.$ Then for $0 < t < p/2,$ there exists a constant $A_{p,t} > 0$ (depending only on $p$ and $t$) such that  
$$\|g\circ \sigma_a - g(a)\|_t^t \le  A_{p,t}\|T_g f_a\|_p^t.$$
\end{aleman_cima}
%

\section{Main result}

Our goal is to show that a non-compact operator $T_g\colon H^p \to H^p, \, 1 \le p < \infty, \,  g \in BMOA\setminus VMOA$, fixes an isomorphic copy of $\ell^p$ yielding that the compactness and strict singularity are equivalent for $T_g$ on $H^p$. This is done by constructing bounded operators $V\colon \ell^p \to H^p$ and $U\colon \ell^p \to H^p,$ where $V(\ell^p) = M$ is a closure of a linear span of suitably chosen test functions $f_{a_k} \in H^p$ and the operator $U$ is an isomorphism onto its image $U(\ell^p) = T_g(M)$. Then it is straightforward to show that the restriction $T_g|_M\colon M \to T_g(M)$ is bounded from below by a positive constant and consequently an isomorphism, see Figure \ref{fig}. 

\begin{figure}[h]
\caption{\textbf{Operators $U, V$ and $T_g$}}
\label{fig}
\centering
\begin{tikzpicture}
  \matrix (m) [matrix of math nodes,row sep=3em,column sep=4em,minimum width=2em]
  {
     \ell^p & \\ 
     H^p & H^p \\};
  \path[-stealth]
    (m-1-1) edge node [left] {$V$} (m-2-1)
    (m-1-1) edge node [right] {$U$} (m-2-2)        
    (m-2-1.east|-m-2-2) edge node [below] {$T_g$}
            node [above] {} (m-2-2)
            edge [dashed,-] (m-2-1);
\end{tikzpicture}
\end{figure}
The strategy for choosing the suitable test functions in Proposition \ref{bddmap} and Theorem \ref{isomorphism} is similar to the one used by Laitila, Nieminen and Tylli in \cite{Tylli}, where they utilized these test functions to show that a non-compact composition operator $C_\varphi\colon H^p \to H^p,$ where $\varphi\colon \D \to \D$ is analytic, fixes an isomorphic copy of $\ell^p$.

Before proving our main result (Theorem \ref{T_gstrictsing}), we need some preparations. We prove first a localization lemma for the standard test functions in $H^p, \, 1 \le p < \infty$, defined by $$f_a(z) = \left[\frac{1-|a|^2}{(1-\bar{a}z)^2}\right]^{1/p}, \quad z \in \D,$$ for each $a \in \D$.  
\begin{masslemma}
\label{masslemma}
Let $1 \le p < \infty, \, \varepsilon > 0$ and $(a_k) \subset \D$ be a sequence s.t.\ $(|a_k|)$ is increasing and $a_k \to \omega \in \T.$ Define 
$$A_\varepsilon = \{e^{i\theta}: |\theta - \textup{arg}(\omega)| < \varepsilon\}.$$ 
Then
\begin{align*}
&\textrm{(i) $\lim_{k \to \infty}\int_{\T\setminus A_\varepsilon}|f_{a_k}|^p dm = 0.$}& \\ 
&\textrm{(ii) If $k$ is fixed, then $\lim_{\varepsilon \to 0}\int_{A_\varepsilon}|f_{a_k}|^p dm = 0.$}& 
\end{align*}
\end{masslemma}
\begin{proof}

\textbf{(i)} Fix $\varepsilon > 0.$ It holds that $$|1 - \bar{a_k}\zeta| \gtrsim |1 - \bar{\omega} \zeta| \ge |\omega - \zeta| \gtrsim \varepsilon$$ for $\zeta \in \T \setminus A_\varepsilon$ and large enough $k$. Thus $$|f_{a_k}(\zeta)|^p = \frac{1 - |a_k|^2}{|1-\bar{a_k}\zeta|^2} \le  \frac{1 - |a_k|^2}{|\omega-\zeta|^2} \lesssim \frac{1-|a_k|^2}{\varepsilon^2}$$ when $\zeta \in \T \setminus A_\varepsilon,$ and it follows that $$\lim_{k \to \infty}\int_{\T\setminus A_\varepsilon}|f_{a_k}|^p dm = 0.$$

\textbf{(ii)} Fix $k.$ It follows from the absolute continuity of a measure $B \mapsto \int_B |f_{a_k}|^p dm$ that $\int_{A_\varepsilon} |f_{a_k}|^p dm \to 0$ as $\varepsilon \to 0.$
\end{proof}

Next, utilizing test functions $f_{a_k}, a_k \in \D,$ for which $|a_k| \to 1$ sufficiently fast, we construct a bounded operator $V\colon \ell^p \to H^p$.
\begin{bddmap}
\label{bddmap}
Let $1 \le p < \infty$ and $(a_n) \subset \D$ be a sequence s.t.\ $(|a_n|)$ is increasing and $a_n \to \omega \in \T.$ Then there exists a subsequence $(b_n) \subset (a_n)$ so that the mapping $$S \colon \ell^p \to H^p, \, S(\alpha) = \sum_{n = 1}^\infty \alpha_n f_n,$$ where $\alpha = (\alpha_n) \in \ell^p$ and $f_n = f_{b_n},$ is bounded. In particular, every mapping $$V \colon \ell^p \to H^p, \, V(\alpha) = \sum_{n = 1}^\infty \alpha_n f_{c_n},$$ where $ (c_n) \subset (b_n),$ is bounded.
\end{bddmap}
\begin{proof}
For each $\varepsilon > 0,$ we define a set $A_{\varepsilon} = \{e^{i\theta}: |\theta - \textup{arg}(\omega)| < \varepsilon\}.$ 
Using the fact that $\|f_a \|_p = 1$ for all $a \in \D$ and Lemma \ref{masslemma}, 
we choose positive numbers $\varepsilon_n$ with $\varepsilon_1 > \varepsilon_2 > \ldots > 0$ and numbers $b_n \in (a_n)$ s.t.\ the following conditions hold
\begin{eqnarray*}
&\textup{(i)}& \left(\int_{A_n} |f_j|^p dm \right)^{1/p} < 4^{-n}, \quad j = 1,\ldots, n - 1; \\ 
&\textup{(ii)}& \left(\int_{\T \setminus A_n} |f_n|^p dm \right)^{1/p} < 4^{-n}; \\
\bigg(&\textup{(iii)}& \left(\int_{A_n} |f_n|^p dm \right)^{1/p} \le \|f_n \|_p = 1\bigg)
\end{eqnarray*}  
for every $n \in \N,$ where $A_n = A_{\varepsilon_n}.$

Using conditions (i)-(iii), we show the upperbound $\|S\alpha\|_p \le C \|\alpha\|_{\ell^p}$ for all $\alpha = (\alpha_j) \in \ell^p,$ where $C > 0$ may depend on $p$.

\begin{eqnarray*}
\|S\alpha\|_p^p &=& \int_{\T}\left|\sum_{j=1}^\infty \alpha_j f_j \right|^p dm = \sum_{n = 1}^\infty \int_{A_n \setminus A_{n+1}}\left|\sum_{j=1}^\infty \alpha_j f_j \right|^p dm 
\\
&\le& \sum_{n = 1}^\infty \left( \sum_{j = 1}^\infty |\alpha_j| \left( \int_{A_n \setminus A_{n+1}}|f_j|^p dm \right)^{1/p} \right)^p 
\\
&\le& \sum_{n = 1}^\infty \left( |\alpha_n|\left( \int_{A_n \setminus A_{n+1}}|f_n|^p dm \right)^{1/p}+\sum_{j \ne n} |\alpha_j| \left( \int_{A_n \setminus A_{n+1}}|f_j|^p dm \right)^{1/p} \right)^p,
\end{eqnarray*}
where 
\begin{equation}
\label{eq: est1}
\left( \int_{A_n \setminus A_{n+1}}|f_j|^p dm \right)^{1/p} \le \left( \int_{A_n }| f_j|^p dm \right)^{1/p} < 4^{-n}
\end{equation} 
for $j < n$ by condition (i) and 
\begin{equation}
\label{eq: est2}
\left( \int_{A_n \setminus A_{n+1}}|f_j|^p dm \right)^{1/p} \le \left( \int_{\T \setminus A_j }|f_j|^p dm \right)^{1/p} < 4^{-j} 
\end{equation}
for $j > n$ by condition (ii). Thus by estimates \eqref{eq: est1} and \eqref{eq: est2}, it always holds that 
\begin{equation}
\label{eq: est3}
\left(\int_{A_n \setminus A_{n+1}}|f_j|^p dm \right)^{1/p} < 2^{-n-j}
\end{equation} 
for $j \ne n.$ By using estimate \eqref{eq: est3} we get
\begin{eqnarray*}
\|S(\alpha)\|_p^p &\le& \sum_{n = 1}^\infty \left( |\alpha_n|\left( \int_{A_n \setminus A_{n+1}}|f_n|^p dm \right)^{1/p}+\sum_{j \ne n} |\alpha_j| 2^{-n-j} \right)^p 
\\
&\le& \sum_{n = 1}^\infty (|\alpha_n| + \|\alpha\|_{\ell^p} 2^{-n})^p 
\\
&\le& 2^p \left(\sum_{n = 1}^\infty |\alpha_n|^p + \|\alpha\|_{\ell^p}^p \sum_{n = 1}^\infty 2^{-np}\right) 
= 2^{p+1}\|\alpha\|_{\ell^p}^p,
\end{eqnarray*}
where we also used condition (iii) in the second inequality.

Let $(c_k)$ be a subsequence of $(b_n).$ Then $(c_k) = (b_{n_k})$ for some sequence $0 < n_1 < n_2 < \ldots.$ By considering an isometry $$J\colon \ell^p \to \ell^p, \, (\alpha_k) \mapsto (\beta_j),$$ where $$\beta_j = \begin{cases} \alpha_k, &\mbox{ if } j = n_k \mbox{ for some $k$} \\ 0, &\mbox{ otherwise}, \end{cases}$$ we see that the operator $V = SJ$ is bounded.
\end{proof}
For a non-compact bounded operator $U$ on a Banach space of analytic functions, there exists a weakly (or weak-star in non-reflexive space case) convergent sequence $(g_n)$ so that the sequence $(Ug_n)$ of image points does not converge to zero in norm. The next result states that for a non-compact operator $T_g$ on $H^p$ we can find a sequence $(f_k)$ of test functions converging weakly to zero (or in the weak-star topology for $p=1$) so that the sequence $(T_gf_k)$ converges to a positive constant in norm. The proof is based on Theorem \ref{aleman_cima} of Aleman and Cima.
\begin{normlimit}
\label{normlimit}
Let $g \in BMOA \setminus VMOA$ and $1 \le p < \infty.$ Then there exists a constant $c > 0$ s.t.\ $$\limsup_{|a| \to 1}\|T_g f_a\|_p = c.$$ In particular, there exists a sequence $(a_k) \subset \D$ s.t.\ $$0 < |a_1| < |a_2| < \ldots < 1$$ and $a_k \to \omega \in \T$ so that $$\lim_{k \to \infty}\|T_g f_k\|_p = c.$$ 
\end{normlimit}
\begin{proof}

It follows from Theorem \ref{aleman_cima} that for all $t \in (0,p/2)$ there exists a constant $C = C(p,t) > 0$ s.t.\ 
\begin{equation}
\label{eq: ineq1}
\| T_g f_a\|_p^t \ge C \|g \circ \sigma_a - g(a)\|_t^t 
\end{equation}
for all $a \in \D,$ where $\sigma_a(z) = (a-z)/(1-\bar{a}z).$ 
For each $0 < q < \infty,$ it holds that $$\textup{dist}(g, VMOA) \simeq \limsup_{|a| \to 1}\|g \circ \sigma_a - g(a) \|_q,$$ where the constants of comparison depend on $q$, 
see, e.g.\ Lemma 3 in \cite{LMN}. Thus by choosing $t = p/4$ in \eqref{eq: ineq1} and using Lemma 3 in \cite{LMN} we get 
$$\limsup_{|a| \to 1}\|T_g f_a\|_p \ge C'' \limsup_{|a| \to 1}\|g \circ \sigma_a - g(a)\|_{p/4} \simeq \textup{dist}(g, VMOA) > 0,$$ since $g \in BMOA \setminus VMOA.$ 
Thus there exists a constant $c > 0$ s.t.\ $$\limsup_{|a| \to 1}\|T_g f_a\|_p = c.$$ In particular, by the compactness of $\overline{\D}$ there exists a sequence $(a_k) \subset \D$ s.t.\ $0 < |a_1| < |a_2| < \ldots < 1$ and $a_k \to \omega \in \T$ so that 
$$\lim_{k \to \infty}\|T_g f_k\|_p = c.$$
\end{proof}
The next lemma is a generalization of Lemma 5 in \cite{LMN} for $1 \le p < \infty$.
\begin{localization}
\label{localization}
Let $a \in \D, \, 1 \le p < \infty, \, g \in BMOA$ and $$f_a(z) = \frac{(1-|a|)^{1/p}}{(1-\bar{a}z)^{2/p}}, \quad z \in \D.$$ Define $$I(a) = \left\{e^{i\theta}: |\theta - \textup{arg}(a)| < (1 - |a|)^{\frac{1}{2(2+p)}}\right\}.$$ Then $$\lim_{|a| \to 1}\int_{\T \setminus I(a)} |T_gf_a|^p dm = 0.$$
\end{localization}
\begin{proof}
By rotation invariance, we may assume that $a \in (0,1).$ Also, $g(0) = 0.$  It holds that $|1 - a s e^{i\theta}| \ge C |\theta|$ for all $0 \le s < 1$ and $|\theta| \le \pi,$ where $C > 0$ is an absolute constant. Thus for all $0 \le s < 1$ and $(1 - a)^{\frac{1}{2(2+p)}} \le |\theta| \le \pi$ we have $$|f_a(se^{i\theta})|^p \lesssim \frac{1-a}{|1 - a s e^{i\theta}|^2} \lesssim \frac{1-a}{|\theta|^2} \le (1-a)^{1 - \frac{1}{2+p}}$$ and $$|f_a'(a s e^{i\theta})|^p \lesssim \frac{1 - a}{|1 - a s e^{i\theta}|^{2+p}} \lesssim \frac{1-a}{|\theta|^{2+p}} \le (1-a)^{1/2}.$$ For $\zeta \in \T \setminus I(a)$, we obtain 
\begin{eqnarray*}
|T_g f_a (\zeta)|^p &=& \left|\int_0^1 f_a(s \zeta)g'(s \zeta) \zeta ds \right|^p 
\\
&\le& 2^p \left( |f_a(\zeta)g(\zeta)|^p + \left(\int_0^1 |f_a'(s\zeta)g(s\zeta)|ds\right)^p \right) 
\\
&\lesssim& (1-a)^{1-\frac{1}{2+p}}|g(\zeta)|^p + (1 - a)^{1/2} \left(\int_0^1 |g(s\zeta)|ds\right)^p. 
\end{eqnarray*}
Since $g \in BMOA \subset \mathcal{B},$ it holds that $|g(z)| \lesssim \log\left( \frac{1}{1 - |z|}\right)$ and consequently $\int_0^1 |g(s\zeta)|ds \lesssim \|g \|_{*},$ where $C > 0$ is an absolute constant and $\|g \|_{*} = \sup_{a \in \D}\|g \circ \sigma_a - g(a) \|_2.$ Therefore $$\int_{\T \setminus I(a)} |T_gf_a|^p dm \lesssim (1 - a)^{1 - \frac{1}{2+p}} \|g\|_p^p + (1 - a)^{1/2} \|g\|_{*}^p \to 0$$ as $a \to 1,$ where $\|g\|_p \le \sup_{a \in \D}\|g \circ \sigma_a - g(a) \|_p \simeq \|g \|_{*}$.
\end{proof}
Using Lemma \ref{localization}, we prove the following localization result for the images $T_gf_a, \, a \in \D,$ of the test functions $f_a$ (cf. Lemma \ref{masslemma}).
\begin{localization2}
\label{localization2}
Let $(a_k) \subset \D$ be s.t.\ $0 < |a_1| < |a_2| < \ldots < 1$ and $a_k \to \omega \in \T.$  Define $$A_\varepsilon = \{e^{i\theta}: |\theta - \textup{arg}(\omega)| < \varepsilon\}$$ for each $\varepsilon > 0$ and $f_k = f_{a_k}$. Then
\begin{align*}
&\textrm{(i) $\lim_{k \to \infty}\int_{\T\setminus A_\varepsilon}|T_g f_k|^p dm = 0$ for every $\varepsilon > 0.$}& \\ 
&\textrm{(ii) If $k$ is fixed, then $\lim_{\varepsilon \to 0}\int_{A_\varepsilon}|T_g f_k|^p dm = 0.$}& 
\end{align*}

\end{localization2}
\begin{proof}
\textbf{(i)} Let $\varepsilon > 0$. Since $a_k \to \omega,$ we have $|\textup{arg}(a_k) - \textup{arg}(\omega)| < \frac{\varepsilon}{2}$ and $(1 - |a_k|)^{\frac{1}{2(2+p)}} < \frac{\varepsilon}{2}$ for $k$ large enough. Consequently we have $$I(a_k) = \left\{e^{i\theta}: |\theta - \textup{arg}(a_k)| < (1 - |a_k|)^{\frac{1}{2(2+p)}}\right\}\subset A_\varepsilon$$ for $k$ large enough. Thus by Lemma \ref{localization} $$\int_{\T\setminus A_\varepsilon}|T_g f_k|^p dm \le \int_{\T \setminus I(a_k)} |T_g f_k|^p dm \to 0$$ as $k \to \infty.$

\textbf{(ii)} If $k$ is fixed, then it follows from the absolute continuity of a measure $B \mapsto \int_B |T_g f_k|^p dm$ that $\int_{A_\varepsilon} |T_g f_k|^p dm \to 0$ as $\varepsilon \to 0.$
\end{proof}
As a final step before the proof of Theorem \ref{T_gstrictsing}, we construct an isomorphism $U\colon \ell^p \to H^p$ using a non-compact $T_g$ and test functions.
\begin{isomorphism}
\label{isomorphism}
Let $g \in BMOA \setminus VMOA, \, 1 \le p < \infty$ and $(a_n) \subset \D$ be the sequence from Proposition \ref{normlimit}. Then there exists a subsequence $(b_n) \subset (a_n)$ s.t.\ the mapping $$U\colon \ell^p \to H^p,\, U(\alpha)=\sum_{n = 1}^\infty \alpha_n T_g f_n,$$ where $\alpha = (\alpha_n) \in \ell^p$ and $f_n = f_{b_n}$, is an isomorphism onto its image.
\end{isomorphism}
\begin{proof}
We need to show that $\|U(\alpha)\|_p \simeq \|\alpha\|_{\ell^p}$ for all $\alpha=(\alpha_n) \in \ell^p.$  
By Proposition \ref{bddmap} there exists a subsequence $(c_n) \subset (a_n)$ inducing a bounded operator $$S\colon \ell^p \to H^p, \, S(\alpha) = \sum_{n = 1}^\infty \alpha_n f_{c_n}$$  
and for any subsequence $(b_n)$ of $(c_n)$ the operator $$V\colon \ell^p \to H^p, \, V(\alpha) = \sum_{n = 1}^\infty \alpha_n f_{b_n}$$ is bounded. 
Therefore the upperbound $\lesssim$ follows from Proposition \ref{bddmap} and the boundedness of the operator $T_g$: 
\begin{eqnarray}
\label{eq: bddness}
\left\|T_g\left(\sum_{n = 1}^\infty \alpha_n f_{b_n}\right)\right\|_p &\le& \|T_g\|_{H^p \to H^p} \left\|\sum_{n = 1}^\infty \alpha_n f_{b_n} \right\|_p 
= \|T_g\|_{H^p \to H^p} \|V(\alpha)\| 
\nonumber \\
&\lesssim& \|T_g\|_{H^p \to H^p} \|\alpha\|_{\ell^p},
\end{eqnarray}
where $(b_n)$ is any subsequence of $(c_n).$ 

Before proving the lowerbound $\gtrsim,$ we make some preparations. 
Since $(c_n) \subset (a_n),$ it holds that $c_n \to \omega \in \T$ and there exists a constant 
$c > 0$ s.t.\ $\lim_{n \to \infty}\|T_g f_{c_n}\|_p = c$ by Proposition \ref{normlimit}. For each $\varepsilon > 0,$ we define a set $A_\varepsilon = \{e^{i\theta}: |\theta- \textup{arg}(\omega)| < \varepsilon\}.$ Also, we define sequences $(\varepsilon_n)$ 
and $(b_n) \subset (c_n)$ inductively using Proposition \ref{normlimit} and Lemma \ref{localization2} in the following way: 

We choose positive numbers $\varepsilon_n$ and $b_n \in (c_n)$ with $$\varepsilon_1 > \varepsilon_2 > \ldots > 0$$ s.t.\ the following conditions hold
\begin{eqnarray*}
&\textup{(i)}& \left(\int_{A_n} |T_g f_j|^p dm \right)^{1/p} < 4^{-n} \delta c, \quad j = 1,\ldots, n - 1; \\ 
&\textup{(ii)}& \left(\int_{\T \setminus A_n} |T_g f_n|^p dm \right)^{1/p} < 4^{-n} \delta c; \\
&\textup{(iii)}& \frac{c}{2} \le \left(\int_{A_n} |T_g f_n|^p dm \right)^{1/p} \le 2 c
\end{eqnarray*}  
for every $n \in \N,$ where $A_n = A_{\varepsilon_n}, \, f_n = f_{b_n}$ and $\delta > 0$ is a constant whose value is determined later.

Now we are ready to prove the lower estimate $\|U\alpha\|_p \ge C \|\alpha\|_{\ell^p},$ where the constant $C > 0$ may depend on $p.$ 
\begin{eqnarray*}
&&\|U\alpha\|_p^p = \int_{\T}\left|\sum_{j=1}^\infty \alpha_j T_g f_j \right|^p dm = \sum_{n = 1}^\infty \int_{A_n \setminus A_{n+1}}\left|\sum_{j=1}^\infty \alpha_j T_g f_j \right|^p dm
\\
&\ge& \sum_{n = 1}^\infty \left(|\alpha_n| \left(\int_{A_n \setminus A_{n+1}}| T_g f_n|^p dm\right)^{1/p} 
-\sum_{j \ne n}|\alpha_j|\left(\int_{A_n \setminus A_{n+1}}| T_g f_j|^p dm\right)^{1/p} \right)^p,
\end{eqnarray*}
where $$\left( \int_{A_n \setminus A_{n+1}}|T_g f_j|^p dm \right)^{1/p} \le \left( \int_{A_n }|T_g f_j|^p dm \right)^{1/p} < 4^{-n}\delta c$$ for $j < n$ by condition (i) and $$\left( \int_{A_n \setminus A_{n+1}}|T_g f_j|^p dm \right)^{1/p} \le \left( \int_{\T \setminus A_j }|T_g f_j|^p dm \right)^{1/p} < 4^{-j}\delta c$$ for $j > n$ by condition (ii). Thus it always holds that $$\left(\int_{A_n \setminus A_{n+1}}|T_g f_j|^p dm \right)^{1/p} < 2^{-n-j}\delta c$$ for $j \ne n.$ Consequently, we can estimate
\begin{eqnarray*}
\|U\alpha\|_p^p &\ge& \sum_{n = 1}^\infty \left(|\alpha_n| \left(\int_{A_n \setminus A_{n+1}}| T_g f_n|^p dm\right)^{1/p} - \sum_{j = 1}^\infty |\alpha_j| 2^{-n-j}\delta c \right)^p 
\\
&\ge&  \sum_{n = 1}^\infty \left(|\alpha_n| \left(\frac{c}{2}- 4^{-n-1}\delta c \right) - \|\alpha\|_{\ell^p} 2^{-n} \delta c \right)^p
\\
&\ge& \sum_{n = 1}^\infty \left( \frac{c}{2}|\alpha_n| - \|\alpha\|_{\ell^p}(4^{-n-1} + 2^{-n})\delta c \right)^p 
\\
&\ge& \sum_{n = 1}^\infty \left( \frac{c}{2}|\alpha_n| - 2^{-n+1} \delta c \|\alpha\|_{\ell^p} \right)^p
\\
&\ge& \sum_{n = 1}^\infty \left( 2^{-p} \left(\frac{c}{2}\right)^p|\alpha_n|^p - 2^{(-n+1)p}\delta^p c^p \|\alpha\|_{\ell^p}^p \right) 
\\
&=& 2^{-2p} c^p \|\alpha\|_{\ell^p}^p - 2^p \delta^p c^p \left(\sum_{n = 1}^\infty 2^{-np}\right) \|\alpha\|_{\ell^p}^p 
\\
&\ge& (2^{-2p}-2\delta^p)c^p \|\alpha\|_{\ell^p}^p = 2^{-2p-1} c^p \|\alpha\|_{\ell^p}^p,
\end{eqnarray*}
when we choose $\delta > 0$ s.t.\ $2^{-2p}-2\delta^p = 2^{-2p-1}$, i.e.\ $\delta = 2^{-2-2/p}.$ Thus the mapping $U$ is bounded from below and by \eqref{eq: bddness} it is also bounded. Therefore we have established that $$\|U(\alpha)\|_p \simeq  \|\alpha\|_{\ell^p}$$ for all $\alpha \in \ell^p$ and consequently the mapping $U$ is an isomorphism onto its image.
\end{proof}

%
Now we are ready to prove our main result.

\begin{proof}[Proof of Theorem \ref{T_gstrictsing}]
By Theorem \ref{isomorphism} and Proposition \ref{bddmap}, we can choose a sequence $(b_n) \subset \D$ that  induces an isomorphism $$U\colon \ell^p \to H^p,\, U(\alpha)=\sum_{n = 1}^\infty \alpha_n T_g f_n$$ onto its image and 
a bounded operator $$V \colon \ell^p \to H^p, \, V(\alpha) = \sum_{n = 1}^\infty \alpha_n f_n,$$ where $f_n = f_{b_n}$ and $\alpha = (\alpha_n) \in \ell^p.$ 

Define $M = \overline{\textup{span}\{f_n\}},$ where the closure is taken in $H^p.$ 
It is enough to show that the restriction $$T_g|_M\colon M \to T_g(M)$$ is bounded from below and $M$ is isomorphic to $\ell^p$. Let $f \in M.$ Then $f = \sum_{n = 1}^\infty \alpha_n f_n$ for some $\alpha = (\alpha_n) \in \ell^p$ and it follows from the fact that $U$ is bounded from below and the boundedness of $V$ that
\begin{eqnarray*}
\|T_gf\|_p &=& \left\|\sum_{n = 1}^\infty \alpha_n T_gf_n \right\|_p = \|U(\alpha)\|_p \gtrsim \|\alpha\|_{\ell^p} \gtrsim \|V(\alpha)\|_p 
\\
&=& \|\sum_{n = 1}^\infty \alpha_n f_n\|_p = \|f\|_p.
\end{eqnarray*} 
Since the operator $T_g|_M$ is also bounded, it is an isomorphism. Moreover, it holds that $\ell^p$ is isomorphic to $ U(\ell^p) = T_g(M)$, which is isomorphic to $M.$ 
Consequently the operator $T_g$ fixes an isomorphic copy of $\ell^p,$ namely the closed subspace $M$. Hence the operator $T_g$ is not strictly singular. 
\end{proof}

\section{Some comments}


It follows from an idea of  Le{\u\i}bov \cite{leibov} that there exists isomorphic copies of the space $c_0$ of null sequences inside $VMOA$. Therefore the strict singularity of $T_g$ on $BMOA$ or on $VMOA$ is equivalent to the compactness of $T_g$ on the same space. The sketch of the proof is the following: 

First, we give a reformulation of  Le{\u\i}bov's result, which is taken from \cite{LNST}.
\begin{lemma}[{\cite[Proposition 6]{LNST}}]\label{lemma_leibov}
Let $(f_n)$ be a sequence in $VMOA$ such that $\|f_n|_* \simeq 1$
and $\|f_n|_2\to 0$ as $n\to \infty$. Then there is a
subsequence $(f_{n_j})$ which is equivalent to the natural basis of
$c_0$; that is, the map
$\iota\colon (\lambda_j) \to \sum_j \lambda_j f_{n_j}$ is an
isomorphism from $c_0$ into $VMOA$.
\end{lemma}
 
For each arc $I \subset \T,$ we write $|I|$ to denote the length of $I$ and define Carleson windows $$S(I) = \{re^{it}: 1 - |I| \le r < 1, t \in I\}$$ and their corresponding base points $u = (1-|I|)e^{i\theta},$ where $\theta$ is the mid-point of $I$. We also consider ``logarithmic $BMOA$'' space $$LMOA = \left\{ g \in H(\D): \sup_{a \in \D}\lambda(a)\|g \circ \sigma_a - g(a)\|_2 < \infty\right\},$$ where $\lambda(a) = \log\left(\frac{2}{1-|a|}\right).$ The condition $g \in LMOA$ characterizes the boundedness of $T_g$ on $BMOA$ and simultaneously on $VMOA$, see \cite{SZ}.

We consider test functions $f_n(z) = \log(1-\bar{u_n}z),$ where $u_n \in \D$ is the base point of the Carleson window $S(I_n)$ and $(I_n)$ is a sequence of arcs of $\T$ s.t\ $I_n \to 0.$ Define $h_n = f_{n+1}-f_n.$ By the proof of Theorem $2$ in \cite{LMN}, it holds that $\|h_n\|_* \simeq 1$ and $\|h_n\|_2 \to 0,$ as $n \to \infty.$ By Lemma \ref{lemma_leibov}, we can pick a subsequence $(h_{n_k}) \subset (h_n)$ which is equivalent to the standard basis $\{e_k\}$ of $c_0.$ If $T_g$ is non-compact on $VMOA,$ by passing to a subsequence if necessary, we can assume that $\|T_gh_{n_k}\|_* > c > 0$ for some constant $c$ for all $k.$ Since $g \in LMOA \subset BMOA,$ the operator $T_g$ is bounded on $H^2$ and consequently $\|T_gh_{n_k}\|_2 \to 0,$ as $k \to \infty.$ Now we apply Lemma \ref{lemma_leibov} again to obtain (by passing to a subsequence, if needed) that $\{T_gh_{n_k}\}$ is equivalent to the natural basis of $c_0.$ Hence $T_g|M,$ where $M = \overline{\textup{span}\{h_{n_k}\}},$ is an isomorphism onto its image and $T_g$ is not strictly singular on $VMOA$ (or on $BMOA$).

\vspace{12pt}

\textbf{Remark.} In Bergman spaces $A^p, \, 1 \le p < \infty$, which are isomorphic to $\ell^p$, see e.g.\ \cite[Chapter 2.A, Theorem 11]{W}, the strict singularity of the operator $T_g$ coincides with the compactness, since all strictly singular operators on $\ell^p$ are compact. 



\subsection*{Acknowledgements}

I would like to thank my advisor Pekka Nieminen, Hans-Olav Tylli and Eero Saksman for useful discussions regarding this topic.   

\bibliographystyle{plain}
\bibliography{StrictsingT_g}

\end{document}